\documentclass{amsart}

\newcommand{\RR}{\mathbb{R}} 
\newcommand{\CC}{\mathbb{C}} 

\newcommand{\HP}{\mathbf{H}}

\newcommand{\RE}{\mathrm{Re}\,}

\newcommand{\dil}{\mathcal{D}} 
\newcommand{\tran}{\mathcal{T}} 
\newcommand{\rot}{\mathcal{R}}

\newcommand{\Para}{\mathcal{P}}
\newcommand{\Hype}{\mathcal{H}}

\newcommand{\vW}{\mathcal{W}} 

\newcommand{\dc}{d^c}

\newcommand{\paren}[1]{\left(#1\right)}
\newcommand{\abs}[1]{\left\lvert#1\right\rvert}
\newcommand{\norm}[1]{\left\|#1\right\|}

\newcommand{\pd}[2]{\frac{\partial#1}{\partial#2}}
\newcommand{\spd}[3]{\frac{\partial^2#1}{\partial#2\partial#3}}
\newcommand{\pdl}[2]{\partial#1/\partial#2}

\newtheorem{theorem}{Theorem}[section]
\newtheorem{corollary}[theorem]{Corollary}
\newtheorem{lemma}[theorem]{Lemma}

\theoremstyle{definition}

\theoremstyle{remark}

\numberwithin{equation}{section}

\title[A certain K\"ahler potential of the Poincar\'e metric]{A certain K\"ahler potential of the Poincar\'e metric and its characterization}

\author{Young-Jun Choi, Kang-Hyurk Lee and Sungmin Yoo}

\address{Department of Mathematics, Pusan National University, 2, Busandaehak-ro 63beon-
gil, Geumjeong-gu, Busan 46241, Republic of Korea}

\email{youngjun.choi@pusan.ac.kr}

\address{Department of Mathematics and Research Institute of Natural Science, Gyeongsang National University, Jinju, Gyeongnam, 52828, Korea}

\email{nyawoo@gnu.ac.kr}

\address{Center for Geometry and Physics, Institute for Basic Science (IBS), Pohang 37673, Korea}

\email{sungmin@ibs.re.kr}



\thanks{The research of first and second named authors  was supported by the National Research Foundation of Korea (NRF) grant funded by the Korea government (No. 2018R1C1B3005963, No. NRF-2019R1F1A1060891). The last named author was supported by IBS-R003-D1.}

\begin{document}

\maketitle

\begin{abstract}
We will show a rigidity of a K\"ahler potential of the Poincar\'e metric with a constant length differential.
\end{abstract}


\section{Introduction}

From the fundamental result of Donnelly-Fefferman~\cite{DF}, the vanishing of the space of $L^2$ harmonic $(p,q)$ forms has been an important research theme in the theory of complex domains. Since M.~Gromov (\cite{Gromov}, see also \cite{Donnelly1994}) suggested the concept of the K\"ahler hyperbolicity and gave a connection to the vanishing theorem, there have been many studies on the K\"ahler hyperbolicity of the Bergman metric, which is a fundamental K\"ahler structure of bounded pseudoconvex domains. The K\"ahler structure $\omega$ is \emph{K\"ahler hyperbolic} if there is a global $1$-form $\eta$ with $d\eta=\omega$ and $\sup\norm{\eta}_\omega<\infty$.

In \cite{Donnelly1997}, H. Donnelly showed the K\"ahler hyperbolicity of  Bergman metric on some class of  weakly pseudoconvex domains. For bounded homogeneous domain $D$ in $\CC^n$ and its Bergman metric $\omega_D$ especially, he used a classical result of Gindikin~\cite{Gindikin} to show that $\sup\norm{d\log K_D}_{\omega_D}<\infty$. Here $K_D$ is the Bergman kernel function of $D$ so $\log K_D$ is a canonical potential of $\omega_D$.

In their paper \cite{Kai-Ohsawa}, S.~Kai and T.~Ohsawa gave another approach. They proved that every bounded homogeneous domain  has a K\"ahler potential of the Bergman metric whose differential has a constant length.
\begin{theorem}[Kai-Ohsawa \cite{Kai-Ohsawa}]\label{thm:KO1}
For a bounded homogeneous domain $D$ in $\CC^n$, there exists a positive real valued function $\varphi$ on $D$ such that $\log\varphi$ is a K\"ahler potential of the Bergman metric $\omega_D$ and $\norm{d\log\varphi}_{\omega_D}$ is constant.
\end{theorem}

 It can be obtained by the facts that each homogeneous domain is biholomorphic to a Siegel domain  (see \cite{VGP}) and a homogeneous Siegel domain is affine homogeneous (see \cite{KMO}).

More precisely, let us consider a bounded homogeneous domain $D$ in $\CC^n$ and a biholomorphism $F:D\to S$  for a Siegel domain $S$. For the Bergman kernel function $K_S$ of $S$ which is a canonical potential of the Bergman metric $\omega_S$, it is easy to show that $d\log K_S$ has a constant length with respect to $\omega_S$ from the affine homogeneity of $S$  (the group of affine holomorphic automorphisms acts transitively on $S$).  Since $\log K_S$ is a K\"ahler potential of $\omega_S$, the transformation formula of the Bergman kernel implies that the pullback $F^*\log K_S=\log K_S\circ F$ is also a K\"ahler potential of $\omega_D$. Using the fact that $F:(D,\omega_D)\to(S,\omega_S)$ is an isometry, we have $\norm{d(F^*\log K_S)}_{\omega_D}=\norm{d\log K_S}_{\omega_S}\circ F$.  As a function $\varphi$ in Theorem~\ref{thm:KO1}, we can choose the pullback $K_S\circ F$ of  the Bergman kernel function of the Siegel domain.

At this junction,  it is natural to ask:
\begin{quote}
\textit{If there is a K\"ahler potential $\log\varphi$ with a constant $\norm{d\log\varphi}_{\omega_D}$, is it always obtained by the pullback of the Bergman kernel function of the Siegel domain?}
\end{quote}
The aim of this paper is to discuss of this question in the $1$-dimensional case. 

The only bounded homogeneous domain in $\CC$ is the unit disc $\Delta=\{z\in\CC:\abs{z}<1\}$  up to the biholomorphic equivalence and  the $1$-dimensional correspondence of the Bergman metric, namely a holomorphically invariant hermitian structure, is only the Poincar\'e metric. Hence the main theorem as follows gives a positive answer to the question.
\begin{theorem}\label{thm:main thm rough}
Let $\omega_\Delta$ be the Poincar\'e metric of the unit disc $\Delta$. Suppose that there exists a positive real valued function $\varphi:\Delta\to\RR$ such that $\log\varphi$ is a K\"ahler potential of the Poincar\'e metric and $\norm{d\log\varphi}_{\omega_\Delta}$ is constant on $\Delta$. Then $\varphi$ is the pullback of the canonical potential on the half-plane $\HP=\{z\in\CC: \RE z<0\}$.
\end{theorem}

Note that $1$-dimensional Siegel domain is just the half-plane. We will introduce the Poincar\'e metric and related notions in Section~\ref{sec:2}. As an application of the main theorem, we can characterize the half-plane by the canonical potential.

\begin{corollary}\label{cor:main cor rough}
Let $D$ be a simply connected, proper domain in $\CC$ with a Poincar\'e metric $\omega_D=i\lambda dz \wedge d\bar z$. If $\norm{d\log \lambda}_{\omega_D}$ is constant on $D$, then $D$ is affine equivalent to the half-plane $\HP=\{z\in\CC: \RE z<0\}$.
\end{corollary}

In Section~\ref{sec:2}, we will introduce notions and concrete version of the main theorem. Then we will study the existence of a nowhere vanishing complete holomorphic vector field which is tangent to a potential whose differential is of constant length (Section~\ref{sec:3}). Using relations between complete holomorphic vector fields and model potentials in Section~\ref{sec:4}, we will prove theorems.


\section{Background materials}\label{sec:2}

Let $X$ be a Riemann surface. The Poincar\'e metric of $X$ is a complete hermitian metric with a constant Gaussian curvature, $-4$. The Poincar\'e metric exists on $X$ if and only if $X$ is a quotient of the unit disc. If $X$ is covered by $\Delta$, the Poincar\'e metric can be induced by the covering map $\pi:\Delta\to X$ and it is uniquely determined. Throughout of this paper, the K\"ahler form of the Poincar\'e metric of $X$, denoted by $\omega_X$,  stands for the metric also.  When $\omega_X=i\lambda dz\wedge d\bar z$ in the local holomorphic coordinate function $z$, the curvature can be written by
\begin{equation*}
\kappa=-\frac{2}{\lambda}\spd{}{z}{\bar z} \log \lambda \;.
\end{equation*}
So the curvature condition $\kappa\equiv -4$ implies that
\begin{equation*}
\spd{}{z}{\bar z} \log \lambda = 2\lambda \;,
\end{equation*}
equivalently
\begin{equation*}
d\dc\log\lambda = 2\omega_X \;,
\end{equation*}
where $d^c=\frac{i}{2}(\overline\partial-\partial)$.
That means the function $\frac{1}{2}\log\lambda$ is a local K\"ahler potential of $\omega_X$. Any other local potential of $\omega_X$ is always of the form $\frac{1}{2}\log\lambda+\log\abs{f}^2$ where $f$ is a local holomorphic function on the domain of $z$.
We call $\frac{1}{2}\log\lambda$ the \emph{canonical potential} with respect to the coordinate function $z$.  For a domain $D$ in $\CC$, the canonical potential of $D$ means the canonical potential with respect to the standard coordinate function of $\CC$. 

Let us consider the Poincar\'e metric $\omega_\Delta$ of the unit disc $\Delta$:
\begin{equation*}
\omega_\Delta=i\frac{1}{\paren{1-\abs{z}^2}^2}dz\wedge d\bar z = i\lambda_\Delta dz\wedge d\bar z \;.
\end{equation*}
The canonical potential $\lambda_\Delta$ satisfies
\begin{equation*}
\norm{d\log\lambda_\Delta}_{\omega_\Delta}^2=\norm{\pd{\log\lambda_\Delta}{z}dz+\pd{\log\lambda_\Delta}{\bar z}d\bar z}_{\omega_\Delta}^2
=\pd{\log\lambda_\Delta}{z}\pd{\log\lambda_\Delta}{\bar z}\frac{1}{\lambda_\Delta}=4\abs{z}^2 \;,
\end{equation*}
so does not have a constant length.
By the same way of Kai-Ohsawa~\cite{Kai-Ohsawa}, we can get a model for $\varphi$ in Theorem~\ref{thm:KO1} for the unit disc,
\begin{equation}\label{eqn:model}
\varphi_\theta(z)=\frac{\abs{1+e^{i\theta}z}^4}{\paren{1-\abs{z}^2}^2} \quad\text{for $\theta\in\RR$}
\end{equation}
as a pullback of the canonical potential $\lambda_\HP=1/\abs{\RE w}^2$ on the left-half plane $\HP=\{w:\RE w<0\}$ by the Cayley transforms (see \eqref{eqn:CT} for instance). The term $\theta$ depends on the choice of the Cayley transform. Since $\log\varphi_\theta=\log\lambda_\Delta+\log\abs{1+e^{i\theta}z}^4$, the function $\frac{1}{2}\log\varphi_\theta$  is a K\"ahler potential. Moreover
\begin{equation*}
\norm{d\log\varphi_\theta}_{\omega_\Delta}^2 \equiv 4 \;.
\end{equation*}
At this moment, we introduce a significant result of Kai-Ohsawa.
\begin{theorem}[Kai-Ohsawa \cite{Kai-Ohsawa}]\label{thm:KO2}
For a bounded homogeneous domain $D$ in $\CC^n$, suppose that there is a K\"ahler potential $\log\psi$ of the Bergman metric $\omega_D$ with a constant $\norm{d\log\psi}_{\omega_D}$, then $\norm{d\log\psi}_{\omega_D}=\norm{d\log\varphi}_{\omega_D}$ where $\varphi$ is as in Theorem~\ref{thm:KO1}.
\end{theorem}
Suppose that  a positively real valued $\varphi$ on $\Delta$ satisfies that $d\dc\log\varphi=2\omega_\Delta$ and $\norm{d\log\varphi}_{\omega_\Delta}^2\equiv c$ for some constant $c$. Theorem~\ref{thm:KO2} implies that $c$ must be $4$. Therefore, we can rewrite Theorem~\ref{thm:main thm rough} by
\begin{theorem}\label{thm:main thm}
If there exists a function $\varphi:\Delta\to\RR$ satisfying
\begin{equation}\label{eqn:basic condition}
d\dc\log\varphi=2\omega_\Delta \quad\text{and}\quad
\norm{d\log\varphi}_{\omega_\Delta}^2\equiv4
\;.
\end{equation}
Then $\varphi=r\varphi_\theta$ as in \eqref{eqn:model} for some $r>0$ and $\theta\in\RR$.
\end{theorem}

Corollary~\ref{cor:main cor rough} can be also written by

\begin{corollary}\label{cor:main cor}
Let $D$ be a simply connected, proper domain in $\CC$ with a Poincar\'e metric $\omega_D=i\lambda dz \wedge d\bar z$. If $\norm{d\log \lambda}_{\omega_D}^2\equiv 4$, then $D$ is affine equivalent to the half-plane $\HP=\{z\in\CC: \RE z<0\}$.
\end{corollary}


\section{Existence of nowhere vanishing complete holomorphic vector field}\label{sec:3}

In this section, we will study an existence of a complete holomorphic tangent vector field on a Riemann surface $X$ which admits a K\"ahler potential of the Poincar\'e metric with a constant length differential. 

By a holomorphic tangent vector field of a Riemann surface $X$, we means a holomorphic section $\vW$ to the holomorphic tangent bundle $T^{1,0}X$. If the corresponding real tangent vector field $\RE \vW=\vW+\overline{\vW}$ is complete, we also say $\vW$ is complete. Thus the complete holomorphic tangent vector field generates a $1$-parameter family of holomorphic transformations.

In this section, we will show that
\begin{theorem}\label{thm:existence}
Let $X$ be a Riemann surface with the Poincar\'e metric $\omega_X$. If there is a function $\varphi:X\to\RR$ with 
\begin{equation}\label{eqn:condition on surface}
d\dc\log\varphi = 2\omega_X 
\quad\text{and}\quad
\norm{d\log\varphi}_{\omega_X}^2\equiv 4
\end{equation}
then there is a nowhere vanishing complete holomorphic vector field $\vW$ such that $(\RE \vW)\varphi\equiv 0$.
\end{theorem}

\begin{proof}
Take a local holomorphic coordinate function $z$ and let $\omega_X=i\lambda dz\wedge d\bar z$. The equation~\eqref{eqn:condition on surface} can be written by
\begin{equation*}
\paren{\log\varphi}_{z\bar z} = 2\lambda
\quad\text{and}\quad
\paren{\log\varphi}_z \paren{\log\varphi}_{\bar z} = 4\lambda
\end{equation*}
Here, $\paren{\log\varphi}_z=\pd{}{z}\log\varphi$, $\paren{\log\varphi}_{\bar z}=\pd{}{\bar z}\log\varphi$ and $\paren{\log\varphi}_{z\bar z}=\spd{}{z}{\bar z}\log\varphi$. This implies that
\begin{align*}
\paren{\varphi^{-1/2}}_{z}
	&=  
	\pd{}{z}\varphi^{-1/2}
	=
	-\frac{1}{2}\varphi^{-1/2} \paren{\log\varphi}_{z} \; ;
\\
\paren{\varphi^{-1/2}}_{z\bar z}
	&=  
	\spd{}{z}{\bar z}\varphi^{-1/2}
	=
	-\frac{1}{2}\varphi^{-1/2} \paren{\log\varphi}_{z\bar z}
	+\frac{1}{4}\varphi^{-1/2}\paren{\log\varphi}_z \paren{\log\varphi}_{\bar z}
	\\
	&=
	-\frac{1}{2}\varphi^{-1/2} \paren{
		\paren{\log\varphi}_{z\bar z}
		-\frac{1}{2}\paren{\log\varphi}_z \paren{\log\varphi}_{\bar z}
		}
	\\
	&= 0 \; .
\end{align*}
Thus we have that the function $\varphi^{-1/2}$ is harmonic so $\paren{\varphi^{-1/2}}_{z}$ is holomorphic.

Let us consider a local holomorphic vector field,
\begin{equation*}
\vW=\frac{i}{\paren{\varphi^{-1/2}}_z}\pd{}{z} 
=\frac{-2i\varphi^{3/2}}{\varphi_z}\pd{}{z} 
=\frac{-2i\varphi^{1/2}}{\paren{\log\varphi}_z}\pd{}{z} 
\;.
\end{equation*}
In any other local holomorphic coordinate function $w$,  we have
\begin{equation*}
\vW=\frac{i}{\paren{\varphi^{-1/2}}_z}\pd{}{z} 
=\frac{i}{\paren{\varphi^{-1/2}}_w\pd{w}{z}}\pd{w}{z}\pd{}{w} 
=\frac{i}{\paren{\varphi^{-1/2}}_w}\pd{}{w} 
\;.
\end{equation*}
so $W$ is globally defined on $X$. Now we will show that $\vW$ satisfies conditions in the theorem.

Since
\begin{equation*}
\norm{\varphi^{-1/2}\vW}_{\omega_X}^2
=\norm{\frac{-2i}{\paren{\log\varphi}_{z}}\pd{}{z} }_{\omega_X}^2
=\frac{4\lambda}{\paren{\log\varphi}_{z}\paren{\log\varphi}_{\bar z}}
=1 \;,
\end{equation*}
the vector field $\varphi^{-1/2}\vW$ has a unit length with respect to the complete metric $\omega_X$, so the corresponding real vector field $\RE \varphi^{-1/2}\vW=\varphi^{-1/2}(\vW+\overline{\vW})$ is complete.
Moreover
\begin{equation*}
(\RE \vW)\varphi  = \frac{-2i\varphi^{3/2}}{\varphi_{z}}\varphi_z  +\frac{2i\varphi^{3/2}}{\varphi_{\bar z}}\varphi_{\bar z} = 0 \;.
\end{equation*}
Hence it remains to show the completeness of $\vW$.
Take any integral curve $\gamma:\RR\to X$ of $\varphi^{-1/2}\RE\vW$. It satisfies 
\begin{equation*}
\paren{\varphi^{-1/2}(\RE \vW)}\circ \gamma = \dot\gamma
\end{equation*}
equivalently
\begin{equation*}
(\RE \vW)\circ \gamma = \paren{\varphi^{1/2}\circ\gamma} \dot\gamma 
\end{equation*}
The condition $(\RE \vW)\varphi\equiv 0$, equivalently $\varphi^{-1/2}(\RE \vW)\varphi\equiv 0$, implies that the curve $\gamma$ is on a level set of $\varphi$ so $\varphi^{1/2}\circ\gamma\equiv C$ for some constant $C$. The curve $\sigma:\RR\to X$ defined by $\sigma(t)=\gamma(Ct)$ satisfies
\begin{equation*}
(\RE \vW)\circ \sigma (t)  = (\RE \vW)(\gamma(Ct)) = C\dot\gamma(Ct) =\dot\sigma(t)
\end{equation*}
This means that $\sigma:\RR\to X$ is the integral curve of $\RE \vW$; therefore $\RE \vW$ is complete. This completes the proof.
\end{proof}


\section{Complete holomorphic vector fields on the unit disc}\label{sec:4}
In this section, we introduce parabolic and hyperbolic vector fields on the unit disc and discuss their relation to the model potential,
\begin{equation}\label{eqn:model0}
\varphi_0=\frac{\abs{1+z}^4}{\paren{1-\abs{z}^2}^2} 
\end{equation}
where it is $\varphi_\theta$ in \eqref{eqn:model} with $\theta=0$.
\subsection{Nowhere vanishing complete holomorphic vector fields from the left-half plane}
On the left-half plane $\HP=\{w\in\CC:\RE w<0\}$, there are two kinds of affine transformations:
\begin{equation*}
\dil_s(w)=e^{2s} w 
\quad\text{and}\quad
\tran_s(w)=w+2is 
\end{equation*}
for $s\in\RR$. Their infinitesimal generators are
\begin{equation*}
\dil= 2w\pd{}{w} 
\quad\text{and}\quad
\tran=2i\pd{}{w}
\end{equation*}
which are nowhere vanishing complete holomorphic vector fields of $\HP$. Note that
\begin{equation}\label{eqn:relation}
(\tran_s)_*\dil = 2(w-2is)\pd{}{w} = \dil-2s\tran \quad\text{and}\quad
(\tran_s)_*\tran =2i\pd{}{w}= \tran
\end{equation}
for any $s$. 

For the Cayley transform $F:\HP\to\Delta$ defined by
\begin{equation}\label{eqn:CT}
\begin{aligned}
F:\HP&\longrightarrow\Delta \\
w&\longmapsto z=\frac{1+w}{1-w} \;,
\end{aligned}
\end{equation}
we can take two nowhere vanishing complete holomorphic vector fields of $\Delta$:
\begin{equation*}
\Hype=F_*(\dil)=(z^2-1)\pd{}{z} 
\end{equation*}
and
\begin{equation*}
\Para=F_*(\tran)=i(z+1)^2\pd{}{z}\;.
\end{equation*}
When we define $\Hype_s=F\circ\dil_s\circ F^{-1}$ and $\Para_s=F\circ\tran_s\circ F^{-1}$, vector fields $\Hype$ and $\Para$ are infinitesimal generators of $\Hype_s$ and $\Para_s$, respectively. Moreover Equation~\eqref{eqn:relation} can be written by
\begin{equation}\label{eqn:relation'}
(\Para_s)_*\Hype = \Hype-2s\Para \quad\text{and}\quad
(\Para_s)_*\Para= \Para \;.
\end{equation}
There is another complete holomorphic vector field $\rot=iz\pdl{}{z}$ generating the rotational symmetry \begin{equation}\label{eqn:rotation}
\rot_s(z)=e^{is}z \;.
\end{equation} 
The holomorphic automorphism group of $\Delta$ is a real $3$-dimension connected Lie group (cf. see \cite{Cartan,Narasimhan}), we can conclude that any complete holomorphic vector field can be  a real linear combination of $\Hype$, $\Para$ and $\rot$. Since $\Hype(-1)=\Para(-1)=0$ and $\rot(-1)=-i\pdl{}{z}$, we have
\begin{lemma}\label{lem:lc}
If $\vW$ is a complete holomorphic vector field of $\Delta$ satisfying $\vW(-1)=0$, then there exist $a,b\in\RR$ with $\vW=a\Hype+b\Para$.
\end{lemma}


\subsection{Hyperbolic vector fields} In this subsection, we will show that the hyperbolic vector field $\Hype$ can not be tangent to a K\"ahler potential with a constant length differential. 

By the simple computation,
\begin{equation*}
\Hype(\log\varphi_0)
	= (z^2-1)\frac{2(1+\bar z)}{(1+z)(1-\abs{z}^2)}
	=2\frac{\abs{z}^2+z-\bar z-1}{(1-\abs{z}^2)}
	\;,
\end{equation*}
we get 
\begin{equation*}
(\RE\Hype)\log\varphi_0 \equiv-4 \;.
\end{equation*}
 That means $\RE\Hype$ is nowhere tangent to $\varphi_0$. Moreover
\begin{lemma}\label{lem:hyperbolic}
Let $\varphi:\Delta\to\RR$ with $d\dc \log\varphi=2\omega_\Delta$ and $\norm{d\log\varphi}_{\omega_\Delta}^2\equiv 4$. If $(\RE\Hype) \log\varphi\equiv c$ for some $c$, then $c=\pm4$.
\end{lemma}
\begin{proof}
Since $d\dc\log\varphi_0=2\omega_\Delta$ also, the function $\log\varphi-\log\varphi_0$ is harmonic; hence we may let $\log\varphi=\log\varphi_0+f+\bar f$ for some holomorphic function $f:\Delta\to\CC$. Then the condition $(\RE\Hype) \log\varphi\equiv c$ can be written by
\begin{equation}\label{eqn:basic identity}
(\RE\Hype)\log\varphi = -4+(z^2-1)f'+(\bar{z}^2-1)\bar f' \equiv c \;.
\end{equation}
This implies that $(z^2-1)f'$ is constant. Thus we can let
\begin{equation}\label{eqn:hmm}
f'=\frac{C}{z^2-1}
\end{equation}
for some $C\in\CC$.
Since
\begin{equation*}
\pd{}{z}\log\varphi=f'+\pd{}{z}\log\varphi_0
= f'+ \frac{2(1+\bar z)}{(1+z)(1-\abs{z}^2)}\;,
\end{equation*}
we have
\begin{multline*}
\norm{d\log\varphi}_{\omega_\Delta}^2=\paren{\pd{}{z}\log\varphi}\paren{\pd{}{\bar z}\log\varphi}\frac{1}{\lambda_\Delta}
\\
=\abs{f'}^2(1-|z|^2)^2
+\frac{2(1+\bar z)(1-\abs{z}^2)}{(1+z)}\bar f'
+\frac{2(1+z)(1-\abs{z}^2)}{(1+\bar z)}f'
+\norm{d\log\varphi_0}_{\omega_\Delta}^2 \;.
\end{multline*}
From the condition $\norm{d\log\varphi}_{\omega_\Delta}^2\equiv 4\equiv\norm{d\log\varphi_0}_{\omega_\Delta}^2$, it follows
\begin{equation*}
\abs{f'}^2(1-|z|^2)^2
=
-\frac{2(1+\bar z)(1-\abs{z}^2)}{(1+z)}\bar f'
-\frac{2(1+z)(1-\abs{z}^2)}{(1+\bar z)}f'
\;,
\end{equation*}
equivalently
\begin{equation}\label{eqn:identity}
\frac{1}{2}\abs{f'}^2(1-|z|^2)
=
-\frac{(1+\bar z)}{(1+z)}\bar f'
-\frac{(1+z)}{(1+\bar z)}f'
\;.
\end{equation}
Applying \eqref{eqn:hmm} to the right side above,
\begin{multline*}
-\frac{(1+\bar z)}{(1+z)}\bar f'
-\frac{(1+z)}{(1+\bar z)}f'
=\frac{(1+\bar z)}{(1+z)}\frac{\bar C}{1-\bar z^2}
+\frac{(1+z)}{(1+\bar z)}\frac{C}{1-z^2}
\\
=\frac{(1+\bar z-z-\abs{z}^2)\bar C + (1-\bar z+z-\abs{z}^2)C}{\abs{1-z^2}^2} \;.
\end{multline*}
Let $C=a+bi$ for $a,b\in\RR$, then
\begin{equation*}
(1+\bar z-z-\abs{z}^2)\bar C + (1-\bar z+z-\abs{z}^2)C
= 2a(1-\abs{z}^2) +2bi(z-\bar z) \;.
\end{equation*}
Now Equation~\eqref{eqn:identity} can be written by
\begin{equation*}
\frac{1}{2}\frac{\abs{C}^2}{\abs{z^2-1}^2}(1-|z|^2)
=\frac{2a(1-\abs{z}^2) +2bi(z-\bar z)}{\abs{1-z^2}^2}
\;,
\end{equation*}
so we have
\begin{equation*}
(\abs{C}^2-4a)(1-\abs{z}^2) =4bi(z-\bar z) 
\end{equation*}
on $\Delta$.
Take $\partial\bar\partial$ to above, we have
\begin{equation*}
\abs{C}^2-4a=0 \;.
\end{equation*}
Simultaneously $b=0$ so $C=a$. Now we have  $a^2=4a$. Such $a$ is $0$ or $4$. If $f'=4/(z^2-1)$, then $c=4$ from \eqref{eqn:basic identity}. If $f'=0$, then $c=-4$.
\end{proof}


\subsection{Parabolic vector fields}
Since
\begin{equation*}
\Para(\log\varphi_0) 
	= i(z+1)^2\frac{2(1+\bar z)}{(1+z)(1-\abs{z}^2)}
	=2i\frac{\abs{1+z}^2}{1-\abs{z}^2}
	\;,
\end{equation*}
we have
\begin{equation*}
(\RE\Para)\log\varphi_0 \equiv 0 \;.
\end{equation*}
That means that the parabolic vector field $\Para$ is tangent to $\varphi_0$. The vector field $\Para$ is indeed the nowhere vanishing complete holomorphic vector field as constructed in Theorem~\ref{thm:existence} corresponding to $\varphi_0$. The main result of this section is the following.

\begin{lemma}\label{lem:parabolic}
Let $\varphi:\Delta\to\RR$ with $d\dc \log\varphi=2\omega_\Delta$ and $\norm{d\log\varphi}_{\omega_\Delta}^2\equiv 4$. If $(\RE\Para) \log\varphi\equiv c$ for some $c$, then $c=0$ and $\varphi=r\varphi_0$ for some $r>0$.
\end{lemma}
\begin{proof}
By the same way in the proof of Lemma~\ref{lem:hyperbolic}, we let $\log\varphi=\log\varphi_0+f+\bar f$ for some holomorphic $f:\Delta\to\CC$. Since
\begin{equation}\label{eqn:basic identity1}
(\RE\Para)\log\varphi = i(z+1)^2f'-i(\bar{z}+1)^2\bar f' \equiv c
\end{equation}
it follows  that $(z+1)^2f'$ is constant. Thus we have
\begin{equation}\label{eqn:hmm1}
f'=\frac{C}{(z+1)^2}
\end{equation}
for some $C\in\CC$.
Since \eqref{eqn:identity} also holds, we can apply \eqref{eqn:hmm1} to the right side of \eqref{eqn:identity} to get
\begin{multline*}
-\frac{(1+\bar z)}{(1+z)}\bar f'
-\frac{(1+z)}{(1+\bar z)}f' 
=
-\frac{(1+\bar z)}{(1+z)}\frac{\bar C}{(\bar z+1)^2}
-\frac{(1+z)}{(1+\bar z)}\frac{C}{(z+1)^2}
\\
=\frac{-\bar C}{\abs{1+z}^2}
+\frac{-C}{\abs{1+z}^2}
=\frac{-\bar C -C}{\abs{1+z}^2}
\end{multline*}
Now Equation \eqref{eqn:identity} is can be written by
\begin{equation*}
\frac{\abs{C}^2}{\abs{z+1}^4}(1-|z|^2)=2\frac{-\bar C -C}{\abs{1+z}^2} 
\end{equation*}
equivalently
\begin{equation*}
\abs{C}^2(1-|z|^2) =-\paren{2\bar C +2C}\abs{1+z}^2\;.
\end{equation*}
Evaluating $z=0$, we have $\abs{C}^2=-2\bar C-2C$. And taking
 $\partial\bar\partial$ to above, we have $-\abs{C}^2=-2\bar C-2C$. It follows that $C=0$ so $f$ is constant. Moreover Equation \eqref{eqn:basic identity1} implies that $c=0$.
\end{proof}


\section{Proof of the main theorem}\label{sec:5}

Now we prove Theorem~\ref{thm:main thm} and Corollary~\ref{cor:main cor}

\bigskip

\noindent\textit{Proof of Theorem~\ref{thm:main thm}.} Let $\varphi:\Delta\to\RR$ be a function with
\begin{equation*}
d\dc\log\varphi=2\omega_\Delta \quad\text{and}\quad
\norm{d\log\varphi}_{\omega_\Delta}^2\equiv 4
\;.
\end{equation*}
By Theorem~\ref{thm:existence}, we can take a nowhere vanishing complete holomorphic vector field $\vW$ with $(\RE \vW)\varphi\equiv0$. Since every automorphism of $\Delta$ has at least one fixed point on $\overline\Delta$ and $\vW$ is nowhere vanishing on $\Delta$, any nontrivial automorphism generated by $\RE \vW$ has no fixed point in $\Delta$ and should have a common fixed point $p$ at the boundary $\partial\Delta$. This means $p$ is a vanishing point of $\vW$. Consider a rotational symmetry $\rot_\theta$ in \eqref{eqn:rotation} satisfying  $\rot_\theta(-1)=p$. We will show that $\varphi\circ\rot_\theta=r\varphi_0$ where $\varphi_0$ is as in \eqref{eqn:model0} and $r>0$. This implies that $\varphi=r\varphi_{-\theta}$.

Now we can simply denote by $\varphi=\varphi\circ\rot_\theta$ and $\vW=(\rot_\theta^{-1})_*\vW$.  Since $-1$ is a vanishing point of $\vW$, Lemma~\ref{lem:lc} implies
\begin{equation*}
\vW=a\Hype+b\Para
\end{equation*}
for some real numbers $a$, $b$. 

Suppose that $a\neq0$. Equation~\eqref{eqn:relation'} implies that
\begin{equation*}
(\Para_s)_* \vW  = (\Para_s)_*(a\Hype+b\Para) 
= a\Hype-2as\Para+b\Para 
= a\Hype+(b-2as)\Para\;.
\end{equation*}
Take $s=b/2a$, then $\widetilde\vW=(\Para_s)_* \vW =a\Hype$. Let $\tilde\varphi=\varphi\circ\Para_{-s}$ for this $s$. Then $\tilde\varphi$ satisfies conditions in Theorem~\ref{thm:main thm} and $(\RE \widetilde\vW)\tilde\varphi\equiv 0$. But Lemma~\ref{lem:hyperbolic} said that $(\RE \widetilde\vW)\tilde\varphi=a(\RE \Hype)\tilde\varphi\equiv \pm4a\tilde\varphi$. It  contradicts to $(\RE \vW)\varphi\equiv 0$ equivalently $(\RE \widetilde\vW)\tilde\varphi\equiv 0$. Thus $a=0$.

Now $\vW=b\Para$. Since $\vW$ is nowhere vanishing already, $b\neq 0$. The condition $(\RE\vW)\varphi\equiv0$ implies $(\RE\Para)\varphi\equiv 0$. Lemma~\ref{lem:parabolic} says that $\varphi=r\varphi_0$ for some positive $r$. This completes the proof. \qed

\bigskip

\noindent\textit{Proof of Corollary~\ref{cor:main cor}.}
Let $D$ be a simply connected proper domain in $\CC$ and let $\omega_D=i\lambda_D dz \wedge d\bar z$ be its Poincar\'e metric with $\norm{d\log \lambda_D}_{\omega_D}^2\equiv 4$. By Theorem~\ref{thm:existence}, there is a nowhere vanishing complete holomorphic vector field $\vW$ with $(\RE\vW)\lambda_D\equiv0$. Take a biholomorphism $G:\Delta\to D$ and let
\begin{equation*}
\varphi=\lambda_D\circ G
\quad\text{and}\quad
\mathcal{Z}=(G^{-1})_*\vW \,.
\end{equation*}
Note that $(\RE\mathcal{Z})\varphi\equiv0$ by assumption. Using the rotational symmetry $\rot_\theta$ of $\Delta$ which is also affine, we may assume that $\mathcal{Z}(-1)=0$ and we will prove that $G$ is a Cayley transform. 

Since $G:(\Delta,\omega_\Delta)\to(D,\omega_D)$ is an isometry, we have $G^*\omega_D=\omega_\Delta$, equivalently 
\begin{equation*}
\varphi=\frac{\lambda_\Delta}{\abs{G'}^2} \;.
\end{equation*}
Moreover $d\log\varphi=d(G^*\log\lambda_D)$ implies that $\norm{d\log\varphi}_{\omega_D}^2=\norm{d(G^*\log\lambda_D)}_{\omega_D}^2\equiv4$. By Theorem~\ref{thm:main thm}, we have
\begin{equation*}
\frac{\lambda_\Delta}{\abs{G'}^2}=\varphi=r\varphi_0=r\lambda_\Delta \abs{1+z}^4 
\end{equation*}
for some positive $r$. This means that $G'=e^{i\theta'}/\sqrt{r}(1+z)^2$ for some $\theta'\in\RR$ so that
\begin{equation*}
G=\frac{e^{i\theta'}}{2\sqrt{r}}\frac{z-1}{z+1}+C
\end{equation*}
Since the function $z\mapsto (z-1)/(z+1)$ is the inverse mapping of the Cayley transform $F:\HP\to\Delta$ in \eqref{eqn:CT}, we have
\begin{align*}
G\circ F:\HP&\to D \\
z&\mapsto  \frac{e^{i\theta'}}{2\sqrt{r}}z+C \;.
\end{align*}
This implies that $D=G(F(\HP))$ is affine equivalent to $\HP$. \qed


\end{document}